\documentclass[]{amsart}

\usepackage[utf8]{inputenc}

\usepackage{fullpage}

\usepackage{amsmath}
\usepackage{amssymb}
\usepackage{latexsym}
\usepackage{amsthm}

\usepackage{boxedminipage}

\usepackage{listings}

\newcommand{\Cn}{\mathbb{C}^n}
\newcommand{\Cd}{\mathbb{C}^d}
\newcommand{\Cm}{\mathbb{C}^m}
\newcommand{\C}{\mathbb{C}}
\newcommand{\CN}{\mathbb{C}^N}

\newcommand{\R}{\mathbb{R}}

\newcommand{\N}{\mathbb{N}}

\newcommand{\dopt}[2]{\frac{\partial #1}{\partial #2}}

\newcommand{\fps}[1]{\C[[#1]]}

\newlength{\extendaxesby}

\DeclareMathOperator{\imag}{Im}

\DeclareMathOperator{\image}{image}

\newtheorem{thm}{Theorem}
\newtheorem{theorem}{Theorem}

\newtheorem{lemma}[thm]{Lemma}

\newtheorem{proposition}[thm]{Proposition}

\theoremstyle{definition}

\newtheorem{definition}[thm]{Definition}
\newtheorem{exa}{Example}
\newtheorem{example}[exa]{Example}
\newtheorem{rem}{Remark}
\newtheorem{remark}[rem]{Remark}

\usepackage{ifpdf}

\ifpdf
\usepackage[pdftex]{graphicx}
\else
\usepackage{graphicx}
\fi

\begin{document}
\title{Formal meromorphic functions on manifolds of finite type}

\author{Robert Juhlin}
\address{Universit\"at Wien, Fakult\"at f\"ur Mathematik, Nordbergstrasse 15, A-1090 Wien, \"Osterreich}
\email{robert.juhlin@univie.ac.at}

\author{Bernhard Lamel}
\address{Universit\"at Wien, Fakult\"at f\"ur Mathematik, Nordbergstrasse 15, A-1090 Wien, \"Osterreich}
\email{bernhard.lamel@univie.ac.at}

\author{Francine Meylan}
\address{University of  Fribourg, Department of Mathematics, CH 1700 Perolles, Fribourg, Suisse}
\email{francine.meylan@unifr.ch}

\dedicatory{Dedicated to Prof. J.J. Kohn on the occasion of his 75th birthday.}

\thanks{The authors were supported by the Austrian Science Fund FWF, project P19667; the third author 
was partially supported by Swiss NSF grant 2100-063464.00/2. This paper was written during 
the third author's visit to Vienna; she wishes to express her gratitude for the hospitality of 
the faculty of mathematics at the University of Vienna.}
\date{\today}

\ifpdf
\DeclareGraphicsExtensions{.pdf, .jpg, .tif}
\else
\DeclareGraphicsExtensions{.eps, .jpg}
\fi

\maketitle

\begin{abstract} It is shown that a real-valued formal meromorphic function on a formal generic 
	submanifold of finite Kohn-Bloom-Graham type is necessarily constant.
\end{abstract}

\section{Introduction}

It is easy to see (and known, see \cite{BER2}) that if $M\subset\CN$ is a connected generic real-analytic CR manifold which
is of finite type in the sense of Kohn
\cite{Ko1} and Bloom-Graham \cite{BG1} at some point $p\in M$, then any meromorphic map
$H\colon U\to \C^m $ defined on a connected neighbourhood of $M$ which satisfies $H(M) \subset E$, where 
$E\subset\C^m$ is a totally real real-analytic submanifold, is necessarily constant.

Let us give  a short proof of this fact. First, we recall the definition of the 
Segre sets $S_p^j$. These are defined inductively. First, we
define the Segre variety $S_p = S_p^1$ for $p\in M$. Let $\rho(Z,\bar Z) = 
(\rho_1(Z,\bar Z), \dots, \rho_d(Z,\bar Z))$ be a (vector-valued) defining function
for $M$ defined in a neighbourhood $U\times \bar U$ of $(p,\bar p)$, i.e.
\[ M\cap U = \{ Z\in U\colon \rho(Z,\bar Z) = 0 \}, \quad d\rho_1\wedge \cdots
d\rho_d \neq 0 \text{ on  U}, \quad \rho(Z,\bar Z) = \bar \rho (\bar Z, Z).\]
With this notation, $S_q^1$ is defined by 
\[ S_q^1 = \{ Z\in U \colon \rho (Z,\bar q) = 0 \}, \quad q\in U, \]
and the $j$-th Segre set $S_p^j$, $j> 1$, is defined inductively by 
\[ S_p^j = \bigcup_{q\in S_p^{j-1}} S_q^1. \]
For consistency, we also put $S_p^0 = \{ p \}$.

We are using the following Theorem, which characterizes finite type in terms of properties of 
the Segre sets:
\begin{theorem}[Baouendi, Ebenfelt and Rothschild \cite{BER2}]\label{thm:mincrit} Let $M\subset \CN$ be a generic real-analytic 
	CR manifold. Then
	$M$ is of finite type at $p\in M$ if and only if there exists an open set $V\subset \CN $
	with $V\subset S_p^{d+1}$.
\end{theorem}

Now assume that $H\colon U\to\C^m$ is a meromorphic map which satisfies $H(M) \subset E$, 
where  $E$ is totally real. First note that since $M$ is of finite type at some point $p$, it 
is of finite type on the complement of a proper real-analytic subvariety $F\subset M$. So there
exists a point $p\in M$ with the property that $M$ is of finite type at $p$ and $H$ is holomorphic
in some neighbourhood of $p$ (because $M$ is generic, it is a set of uniqueness for 
holomorphic functions). We shall prove that in this situation,
$H$ is constant on an open set in $\CN$, and thus constant.

 We can find coordinates $\eta$ in $\C^m$
such that near $H(p)$, $E$ is given by an equation of the form 
$\eta = \varphi (\bar \eta)$. Thus, $H(Z) = \varphi (\bar{H}(\bar Z))$, whenever 
$Z\in M$, and from this we 
have that $H(Z) =\varphi(\bar{H}(\bar\zeta))$ whenever $Z\in S_\zeta$ (restricting to a suitable 
neighbourhood $U$ of $p$). Thus, $H(Z) = \varphi(\bar{H}(\bar p))$ for $Z\in S_p$; since
$p \in S_p$, $H(Z)= H(p)$ for $Z\in S_p$. Now we consider $Z\in S^2_p$. For each such $Z$, 
there is $\zeta \in S_p^1$ with $Z\in S_\zeta^1$. Our equation tells us that 
$H(Z) = \varphi(\bar H (\bar \zeta)) =\varphi(\bar H (\bar p))$, and again, since $p\in S_p^2$, 
$H(Z) = H(p)$ for $Z\in S^2_p$.

Continuing the iteration process like this, we see that $H(Z) = H(p)$ for $Z\in S^j_p$ for $j\in\N$. Since 
$S_p^{d+1}$ contains an open subset of $\CN$ by Theorem~\ref{thm:mincrit}, the identity principle implies that $H(Z)= H(p)$
on $U$. This proves the constancy of such an $H$. 

Our main point in this paper is the extension of this result to the formal category. Here 
we cannot ``move to a good point''. Let us be a bit more specific about the notions which we 
are going to use (and refer the reader to Baouendi, Mir and Rothschild \cite{BMR1} for more information). 
A {\em generic real formal submanifold} $(M,0) \subset (\CN,0)$ of codimension $d$ is given  by its manifold ideal 
\[ \mathcal{I}(M,0) \subset \fps{Z,\zeta}, \]
which satisfies that $\mathcal{I} (M,0)$ can be generated by $d$ functions $\rho_1,\dots, \rho_d$, where 
$\rho_1,\dots,\rho_d$ have the following properties:
\begin{enumerate}
	\item $\overline{\rho_j} (\zeta, Z) = \rho_j (Z,\zeta)$ (the $\rho_j$ are {\em real})
	\item $\rho_{1;Z} (0) \wedge \dots \wedge \rho_{d;Z} (0) \neq 0$, where  $\rho_{j;Z} = \left(\dopt{\rho_j}{Z_1}, 
	\dots, \dopt{\rho_j}{Z_N}\right)$.
\end{enumerate}
A  {\em formal meromorphic map} $H\colon(\CN,0)\to(\Cm,0)$ is 
given by $H= \frac{N}{D}$, where $ D$ is a
formal power series which is not identically zero, and $N\colon(\CN,0) 
\to (\C^m,0)$ is a formal holomorphic map (i.e., $N=(N_1,\dots, N_m)$ where $N_j \in \fps{Z}$).
 In the present context, $(E,0)\subset (\Cm,0)$ is a {\em formal 
totally real manifold} if it is a formal real submanifold which in suitable (formal) holomorphic coordinates 
$\eta\in \C^m$, $(E,0)$ is given by $\imag \eta = 0$--by this we mean that  $\mathcal{I}(E,0)\subset \fps{\eta,\nu}$ can be 
generated by the functions $\frac{1}{2i} (\nu_j - \eta_j)$. 

\begin{remark}
	Usually, a totally real CR-manifold is defined as a CR-manifold of CR-dimension $0$; what we refer to as ``totally real''
	here is usually referred to as ``maximally totally real''. However, in the formal category, every ``totally real'' submanifold is
	automatically equivalent to a ``maximally totally real'' submanifold; this justifies the chosen terminology.
\end{remark}

\begin{definition}
	We say that a formal meromorphic map $H = N/D\colon (\CN,0) \to (\Cm,0)$  maps $(M,0)\subset (\CN,0)$ into the totally 
	real submanifold $(E,0)\subset (\Cm,0)$ (with coordinates as above)
	if for any formal holomorphic map $\gamma (t) = (\gamma_1 (t), \gamma_2 (t)) \colon (\C^{2N-d}, 0) \to (\C^{2N})$ satisfying 
	$\rho(\gamma_1(t), \gamma_2 (t)) = 0$ for every $\rho \in \mathcal{I}(M,0)$
	it holds that 
	\[ N_j(\gamma_1 (t))\bar D(\gamma_2(t)) - \bar N_j (\gamma_2 (t)) D (\gamma_1 (t)) = 0\]
	for every $j=1,\dots,m$. 
\end{definition}

We also recall that a formal generic manifold $(M,0)\subset (\CN,0)$ is of finite type at $0$ if the 
Lie algebra generated by the formal $(1,0)$- and $(0,1)$-vector fields tangent to $(M,0)$, evaluated at
$0$, spans $\CN$. 
We can now state our main result.

\begin{theorem}\label{thm:main}
	Let $(M,0)\subset (\CN,0)$ be a formal generic manifold of finite type, $H\colon
	(\CN,0) \to (\C^m,0)$ a formal meromorphic  map which maps $(M,0)$ into $(E,0)$, 
	where $(E,0)$ is a formal totally real manifold. Then $H$ is constant.
\end{theorem}

We note that the finite type assumption is necessary. Indeed, every manifold of the 
form $M = \tilde M \times E$, where $\tilde M$ is some CR manifold and $E$ is totally real, 
has nonconstant CR maps onto a totally real manifold (the projection onto its second coordinate). 
On the other hand, here is another example, due to J. Lebl:

\begin{example}
  Let $M\subset \C^3$ be given by 
			\[ w_1 = \overline{w_1} e^{i  p |z|^2}, \quad w_2 = \overline{w_2} e^{i q |z|^2}, \]
			for some integers $p$ and $q$.
	Then the function   
	\[ H(z,w_1,w_2) = \frac{w_1^q}{w_2^p} \]
	maps $M$ into $\R$ and is not the restriction of a holomorphic function. Also
	note that  this  function is not even  
	continuous on $M$. Our results imply that {\em no} nonconstant holomorphic choice of projection onto $\R$
	can be made.
\end{example}

\section{Reflection Identities and Consequences}

We shall first show that we can simplify our situation somewhat by choosing "normal" coordinates.
Recall that normal coordinates for a formal generic submanifold $(M,0)\subset (\CN,0)$ means a choice of 
coordinates $(z,w) \in \Cn\times\Cd$ ($d$ being the real codimension of $(M,0)$) together
with formal functions $Q_j(z,\chi,\tau)\in \fps{z,\chi,\tau}$, $j=1,\dots,d$, satisfying 
\[ Q_j (z,0,\tau) = Q_j (0,\chi,\tau) = \tau_j, \quad j=1,\dots,d,\]
such that $w_j - Q_j (z,\chi,\tau)$ generate the manifold ideal associated to $(M,0)$ in 
$\fps{z,w,\chi,\tau}$. We will write $Q=(Q_1,\dots ,Q_d)$, and abbreviate the generating set 
with $w-Q(z,\chi,\tau)$.

We will show that in normal coordinates, a formal meromorphic function $H$ which maps $(M,0)$ into
$(\R,0)$ actually only depends on the transverse variables $w$. To do this, 
we first give a reflection identity which we will use.

\begin{proposition}\label{pro:reflection}
	If $(M,0)\subset (\CN,0)$ is a formal generic submanifold, and $(z,w)$ are normal coordinates for $(M,0)$
	with corresponding generators $w-Q(z,\chi,\tau)$. If $H = \frac{N}{D}\colon (M,0) \to (\R,0)$ is formal meromorphic, and $N$ and $D$ do not have any common factors, then there exists a formal holomorphic function $a(z,\chi,z^1,w)$, with 
	$a(0,0,0,0) = 1$, such that
	\begin{equation}
		\label{e:reflect} 
		\begin{aligned}
			N\left(z,Q\left(z,\chi,\bar{Q} \left( \chi,z^{1},w\right)\right)\right) &= a(z,\chi,z^1,w)
			N \left( z^1,w\right), \\
			D\left(z,Q\left(z,\chi,\bar{Q} \left( \chi,z^{1},w\right)\right)\right) &= a(z,\chi,z^1,w)
			D \left( z^1,w\right).
		\end{aligned}
	\end{equation}
\end{proposition}
\begin{proof} 
 From the definition, we have
\begin{equation}\label{e:pf1}
	\bar{D}(\chi,\tau)  {N} (z,Q(z,\chi,\tau)) = \bar{N} (\chi,\tau) D(z,Q(z,\chi,\tau)).
\end{equation}
Taking the complex conjugate and replacing $z$ by $z^1$ in this equation, we also have that 
\begin{equation}
	\label{e:pf1conj}
	D(z^1,w) \bar{N} (\chi,\bar{Q} (\chi,z^1,w)) = N (z^1,w) \bar{D} (\chi,\bar{Q} (\chi,z^1,w)).
\end{equation}	
We now substitute $ \tau = \bar{Q} (\chi,z^1,w)$ into \eqref{e:pf1} to obtain
\begin{equation}
	\label{e:pf2} 	\bar{D}(\chi,\bar{Q} (\chi,z^1,w))  {N} (z,Q(z,\chi,\bar{Q} (\chi,z^1,w))) = \bar{N} (\chi,\bar{Q} (\chi,z^1,w)) D(z,Q(z,\chi,\bar{Q} (\chi,z^1,w))).
\end{equation}
We now multiply the left (and right, respectively) hand sides of \eqref{e:pf1conj} and \eqref{e:pf2} with each other, 
and after cancelling a common factor of $ \bar{N} (\chi,\bar{Q} (\chi,z^1,w))\bar{D} (\chi,\bar{Q} (\chi,z^1,w))$
we obtain
\begin{equation}
	\label{e:pffinal} D(z^1,w)  {N} (z,Q(z,\chi,\bar{Q} (\chi,z^1,w))) = D(z,Q(z,\chi,\bar{Q} (\chi,z^1,w)))
	N(z^1,w).
\end{equation}
Now, using the fact that $N$ and $D$ do not have any common factors,  unique factorization in the 
ring $\fps{z,\chi,z^1,w}$ implies that there
exists a unit $a(z,\chi,z^1,w)$ such that \eqref{e:reflect} holds. By evaluating \eqref{e:reflect} at
$z= z^1$, and using the reality property $Q(z,\chi,\bar{Q} (\chi,z,w)) = w$, we have 
that $a(z,\chi,z,w) = 1$, so in particular, $a(0,0,0,0) = 1$.
\end{proof}

\begin{lemma}\label{lem:transverse}
	Let $(M,0)\subset (\CN,0)$ be a formal generic submanifold. Assume that $H(Z) = \frac{N(Z)}{D(Z)}$ is 
	a formal meromorphic map sending $(M,0)$ into $(\R,0)$. Then for any choice of normal coordinates $(z,w)$ for 
	$(M,0)$, we have that $H(z,w) = H(0,w)$; in particular, there exist formal functions $\tilde N (w)$ and $\tilde D (w)$ 
	such that $H(z,w) = \frac{\tilde N (w)}{\tilde D (w)}$.
\end{lemma}
\begin{proof}
We use Proposition~\ref{pro:reflection}. Setting $\chi = z^1 = 0$, we see that
\[ N(z,w) = a(z,0,0,w) N(0,w), \quad D(z,w) = a(z,0,0,w) D(0,w). \]
The Lemma follows. 	
\end{proof}

\section{Prolongation of the reflection along Segre maps and proof of Theorem~\ref{thm:main}}

We will denote by 
\[ v^{1} (z,\chi,z^1;w) = Q(z,\chi,\bar{Q}(\chi,z^1,w)); \]
in the usual Segre-map terminology, $v^1(z,\chi,z^1;0)$ is the transversal component of the
second Segre map of $(M,0)$. Since we shall only have use for the Segre-maps of even order, we introduce 
the notation adapted to our setting. We define $S^{(0)} = z$, and for $j\geq 1$
\[ S^{(j)} = (z,\chi,z^1,\chi^1,\dots,z^j), \]
and write $S^{(j)}_k = (z^k,\chi^k,\dots,z^j)$ for $k\leq j$. By Lemma~\ref{lem:transverse}, $H$ does not depend on $z$ and 
we can 
assume that 
\[ H(z,w)=\frac{N(w)}{D(w)}. \] With that notation and our
simplification from Lemma~\ref{lem:transverse}, 
our reflection identity \eqref{e:reflect} now reads 
\begin{equation}
	\label{e:reflect2} 
	\begin{aligned}
		N\left(v^1 (S^{(1)};w)\right) &= a(S^{(1)},w)
		N \left(w\right), \\
		D\left(v^1 (S^{(1)};w)\right) &= a(S^{(1)},w)
		D \left( w\right).
	\end{aligned}
\end{equation}
For $j\geq 1$, we define inductively
\[ v^{j} \left( S^{(j)} ; w\right) = v^1 (z,\chi,z^1; v^{j-1} (S^{(j)}_1 ;w)). \]

We can now state the finite type criterion of Baouendi, Ebenfelt and Rothschild for {\em formal} submanifolds \cite{BER6}, 
for later  reference, 
as follows:

\begin{theorem}\label{thm:fintypecritformal} If $(M,0)$ is of finite type in the sense of 
	Kohn-Bloom-Graham, 
	then there exists a
	$j\geq 1$ such that 
	\[ S^{(j)} \mapsto v^{j} \left(S^{(j)};0\right), \quad (\C^{(2j-1)n},0) \to (\C^d,0),\]
	is of generic full rank $d$. 
\end{theorem}

Thus, if we replace  $w$ by $ v^{j-1} (S^{(j)}_1;w)$ in \eqref{e:reflect2}, we obtain
\[ N\left(v^j(S^{(j)};w)\right) = N\left(v^1 (S^{(1)};v^{j-1} (S^{(j)}_1 ;w))\right) = 
a (S^{(1)};v^{j-1} (S^{(j)}_1 ;w)) N\left(v^{j-1} (S^{(j)}_1 ;w)\right).\]
Applying induction, we see that the following holds:

\begin{lemma}\label{lem:prolongation}
	For every $j\geq 1$, there exists a unit $a_j (S^{(j)},w)$ such that
	\begin{equation}
		\label{e:reflprol} N \left(v^j(S^{(j)};w)\right) = a_j (S^{(j)},w) N (w), \quad 
		 D \left(v^j(S^{(j)};w)\right) = a_j (S^{(j)},w) D (w).
	\end{equation}
\end{lemma}

We can now prove Theorem~\ref{thm:main}: By Theorem~\ref{thm:fintypecritformal}, there exists a 
$j$ such that $v^j(S^{(j)};0)$ is of generic full rank. Assuming that $D(0) = 0$, we see that 
$D(v^j(S^{(j)};0)) =0$. Since $v^j$ is of generic full rank, this implies that $D(w) = 0$; this 
contradiction shows that $D(0) \neq 0$. Hence, we can assume that $H(z,w) = N(w)$ is formal holomorphic, 
and without loss of generality, $N(0) = 0$. The same argument shows that $N(w) = 0$, and 
so, $H$ is constant.

\begin{remark}
	More generally, if we do not assume that $(M,0)$ is of 
	finite type, then we can define the formal variety 
	\[ V_j = \overline{\image( v^{(j)}(S^{(j)};0))}  \cong 
	\{ f \in \fps{w}  \colon f\circ v^{(j)} (S^{(j)};0)  = 0 \}, \]
	and $V=\cup_j V_j$ (which is again a formal variety). The 
	same arguments as above show that 
	$D$, as well as $N$, are constant on $V$. This corresponds to the statement that 
	a real-valued CR meromorphic function is constant along the CR-orbits of $M$. 
\end{remark}

\bibliographystyle{plain}
\bibliography{bibliography}
\end{document}